\documentclass[12pt]{article}
\usepackage{amsmath}
\usepackage{amsfonts}

\usepackage{graphicx}
\usepackage{color}

\def\epsfig#1{}

\newtheorem{theorem}{Theorem}[section]
\newtheorem{claim}[theorem]{Claim}
\newtheorem{corollary}[theorem]{Corollary}
\newtheorem{definition}[theorem]{Definition}
\newtheorem{example}[theorem]{Example}
\newtheorem{lemma}[theorem]{Lemma}
\newenvironment{proof}[1][Proof]{\textbf{#1.} }{\ \rule{0.5em}{0.5em}}

\newcommand\Azero{{\bf (A0)}}
\newcommand\Aone{{\bf (A1)}}

\newcommand\Atwo{{\bf (A2)}}
\newcommand\Athree{{\bf (A3)}}
\newcommand\Afour{{\bf (A4)}}
\newcommand\Athrees{$\bf (A3)_s$}
\newcommand\Afours{$\bf (A4)_s$}

\newcommand{\Hn}{{\mathcal H}^n}
\newcommand{\X}{{M^+}}
\newcommand{\Y}{{M^-}}
\newcommand{\dom}{\mathop{\rm Dom}}

\newcommand{\graph}{\mathop{\rm Graph}}

\newcommand{\spt}{\mathop{\rm spt}}

\newcommand{\Hess}{\mathop{\rm Hess}}

 \newcommand{\0}{\mathbf{0}}
\newcommand{\R}{\mathbf{R}}

\newcommand{\N}{\mathbf{N}}

\begin{document}

\author{
Robert J. McCann\thanks{
Department of Mathematics, University of Toronto, Toronto Ontario M5S 2E4 Canada,
{\tt mccann@math.toronto.edu}}}

\title{A glimpse into the differential topology and geometry of
optimal transport\thanks{
The author is pleased to acknowledge the support of
Natural Sciences and Engineering Research Council of Canada Grants 217006-08.
\copyright 2012 by the author.}}
\date{\today}

\maketitle

\begin{abstract}
This note exposes the differential topology and geometry underlying
some of the basic phenomena of optimal transportation.  It surveys
basic questions concerning Monge maps and Kantorovich measures:
existence and regularity of the former,  uniqueness of the latter, and
estimates for the dimension of its support,  as well as the associated
linear programming duality.  It shows the answers to these questions concern
the differential geometry and topology of the chosen transportation cost.
It also establishes new connections --- some heuristic and others rigorous ---
based on the properties of the cross-difference of this cost, and its Taylor
expansion at the diagonal.
\end{abstract}


\section{Introduction}


What is optimal transportation?
This subject, reviewed by Ambrosio and Gigli \cite{AmbrosioGigli11p},
McCann and Guillen \cite{McCannGuillen10p},
Rachev and Ruschendorf \cite{RachevRuschendorf98},
and Villani \cite{Villani03} \cite{Villani09} among others,
has become a topic of much scrutiny in recent years,
driven by applications both within and outside mathematics.
However,  the problem has also
lead to the development of its own theory, in which a number of challenging
questions arise and some fascinating answers have been discovered.
The present manuscript is intended to reveal some of the differential topology and
geometry underlying these questions, their solution and variants, and give some novel
and simple yet powerful heuristics for a few highlights from the literature that we survey.
It attempts to frame the phenomenology of the subject, without
delving deeply into many of the methodologies --- both novel and standard ---
which are used to pursue it.
The new heuristics are largely based on the properties of the
cross-difference \eqref{cross-difference},
and its Taylor expansion \eqref{Taylor} at the diagonal.

Given Borel probability measures $\mu^\pm$
on complete separable metric spaces $M^\pm$,
and a continuous bounded function $c(x,y)$
representing the cost per unit mass transported from $x \in M^+$ to $y\in M^-$,
the basic question is to correlate the measures $\mu^+$ and $\mu^-$ so as to
minimize the total transportation cost.  In Monge's 1781 formulation \cite{Monge81},
we seek to minimize
\begin{equation}\label{Monge cost}
cost(G) := \int_{M^+} c(x,G(x)) \,d\mu^+(x)
\end{equation}
among all Borel maps $G: M^+ \longrightarrow M^-$ pushing $\mu^+$ forward to $\mu^-=G_\#\mu^+$,
where the pushed-forward measure is defined by $G_\#\mu^+(Y) = \mu^+(G^{-1}(Y))$ for each
$Y \subset M^-$.  This question is interesting,  because it leads to canonical ways to
reparameterize one distribution of mass with another.  When
the probability measures are given by densities $d\mu^\pm(x) = f^\pm(x)dx$
on manifolds $M^\pm$, we can expect $G$ to satisfy the Jacobian equation
$\pm \det[\partial G^i/\partial x^j]= f^+(x) / f^-(G(x))$.  Additional desirable properties
of $G$ 
can sometimes be guaranteed by a suitable choice of transportation cost;
for example, $G$ will be irrotational for the quadratic cost
$c(x,y) = \frac{1}{2}|x-y|^2$ on Euclidean space \cite{Brenier91}.
For subsequent purposes, we will often assume the cost $c(x,y)$ and
manifolds $M^\pm$ to be smooth,  but quite general otherwise.

In Kantorovich's 1942 formulation,  we seek to minimize
\begin{equation}\label{Kantorvich cost}
cost(\gamma) := \int_{M^+ \times M^-} c(x,y) \,d\gamma(x,y)
\end{equation}
over all joint measures $\gamma \ge 0$ on $M^+ \times M^-$ having $\mu^+$ and $\mu^-$
as marginals.  The form of the latter problem ---
minimize the linear functional $cost(\gamma)$
on the convex set $\Gamma(\mu^+,\mu^-) := \{\gamma \ge 0 \mid \pi^\pm_\#\gamma =\mu^\pm\}$,
where $\pi^+(x,y) = x$ and $\pi^-(x,y) = y$ --- makes it easy to show the Kantorovich infimum
is attained. 
A result of Pratelli \cite{Pratelli07} following Ambrosio 
and Gangbo 
asserts that its value coincides with the
Monge infimum
\begin{equation}\label{K min = M inf}
\min_{\gamma \in \Gamma(\mu^+,\mu^-)} cost(\gamma) = \inf_{\mu^- = G_\# \mu^+ } cost(G)
\end{equation}
if $c$ is continuous and $\mu^+$ is free of atoms.
However it is not straightforward to establish uniqueness of the Kantorovich minimizer,
nor whether the Monge infimum is attained,
and if so, whether the mapping $G$ which attains it is continuous.  A sufficient
condition \Aone$'_+$ for existence (and uniqueness) of optimizers $G$ (and  $\gamma$)
was found by Gangbo \cite{Gangbo95} and Levin \cite{Levin99},
building on work of many authors, including
Brenier, Caffarelli, Gangbo, McCann, Rachev and R\"uschendorf.
When $M^\pm \subset \R^n$ are the closures of open domains,
sufficient conditions  for the existence of a {\em smooth} minimizer
$G:M^+ \longrightarrow M^-$ were provided by Ma, Trudinger and Wang \cite{MaTrudingerWang05},
building on work of Delano\"e, Caffarelli, Urbas and Wang, and later refined
through work of Delano\"e, Figalli, Ge, Kim, Liu, Loeper, McCann,
Rifford, Trudinger, Villani and Wang,
among others.
See Appendix \ref{A:MTW} for a statement of their conditions \Azero$'$-\Afour$'$.
 At the same time,
we introduce a new but equivalent formulation of conditions \Azero$'$-\Afour$'$ in
terms of the cross-difference \eqref{cross-difference}, which emphasizes
their purely topological \Azero-\Atwo\ and geometric \Athree-\Afour\ nature,
exposing their naturality and relevance. 
This process of reformulation, begun with Kim in \cite{KimMcCann10}, is completed here,
as part of a series of questions and responses.

\section{Why do Kantorovich minimizers concentrate on low-dimensional sets?}

Abstractly,  one expects a linear functional $cost(\gamma)$ on a convex set
$\Gamma(\mu^+,\mu^-)$ to attain its infimum at one of the extreme points.  So it
is interesting to understand the extreme points of $\Gamma(\mu^+,\mu^-)$.
Such extreme points are sometimes called {\em simplicial} measures.
Despite much progress, surveyed in \cite{AhmadKimMcCann11}, a characterization
of simplicial measures in terms of their support has long remained elusive and is probably
too much to hope for.  Recall
that a measure $\gamma\ge0$ is {\em simplicial} if it is not
the midpoint of any segment in $\Gamma(\pi^+_\#\gamma,\pi^-_\#\gamma)$.    Ahmad, Kim and
McCann \cite{AhmadKimMcCann11} showed each simplicial measure $\gamma$
vanishes outside the union of a graph
$\{(x,G(x)) \mid x \in M^+\}$
and an antigraph $\{(H(y),y) \mid y \in M^-\}$,
generalizing Hestir and Williams \cite{HestirWilliams95}
result from the special case of Lebesgue measure $\mu^\pm$ on the unit interval
$M^\pm=[0,1]$.
This shows $\gamma$ concentrates on a set whose topological dimension should not exceed
$\max\{n^+,n^-\}$, where $n^\pm = \dim M^\pm$.
Taking $n^+ \le n^-$ without loss of generality,  if the measure
$\mu^-$ fills the space $M^-$,
then $\gamma$ cannot concentrate on any subset of lower dimension than $n^-$,  so
it would seem we have identified the topological dimension of the set on which $\gamma$
concentrates to be precisely $n^- = \max\{n^+,n^-\}$.  Unfortunately,  this simple
argument is somewhat deceptive.  Although the graph and antigraph
of \cite{HestirWilliams95} and \cite{AhmadKimMcCann11} enjoy further structure,
they are not generally $\gamma$-measurable; a priori it is conceivable that
their closures might actually fill the product space $N=M^+ \times M^-$.

With some assumptions on the topology of the cost function $c$ and spaces $M^\pm$,
it is possible to better estimate the size of support of the particular extreme
points of interest using the more robust
notion of Hausdorff dimension.
The basic object of geometrical relevance 
will be the support $\spt \gamma$ of the Kantorovich optimizer,  defined as the smallest closed
subset $S \subset M^+ \times M^+$ carrying the full mass of $\gamma$.
If the Monge infimum \eqref{K min = M inf} is attained by a map
$G:M^+ \longrightarrow M^-$
and the Kantorovich minimizer is unique,  it will turn out that $\spt \gamma$
agrees ($\gamma$-a.e.) with the graph of $G$;  when this map is a diffeomorphism,
then $\gamma$ concentrates
on a subset of dimension $n^+= \dim M^+$ in $M^+ \times M^-$. We shall show why this
might be expected more generally,  assuming $M^\pm$ to be (smooth) manifolds henceforth.

Setting $N := M^+ \times M^-$ and $S= \spt \gamma$,  consider the cross-difference
 \cite{McCann99}
\begin{equation}\label{cross-difference}
\delta(x,y; x_0,y_0) := c(x,y_0) + c(x_0,y) - c(x,y) - c(x_0,y_0)
\end{equation}
defined on $N^2$.
An observation --- special cases of which date back to Monge --- asserts
$\delta(x,y; x_0,y_0) \ge 0$ on $S^2 \subset N^2$;  in other words,
we cannot lower the cost by exchanging partners between $(x,y)$ and $(x_0,y_0)$;
for a modern proof,  see Gangbo and McCann \cite{GangboMcCann96}.
This fact is called the {\em $c$-monotonicity} of $S$.

If $c \in C^2$, then $(x_0,y_0) \in N$ is a critical point for the function
$\delta^0(x,y) := \delta(x,y;x_0,y_0)$,  whose Hessian
\begin{equation}\label{h=Hessian}
h = \textstyle \frac{1}{2} \Hess\nolimits \delta^0(x_0,y_0)
\end{equation}
is then well-defined (though it need not be at points $(x,y) \ne (x_0,y_0)$
which are non-critical).
Now for $(x_0,y_0) \in S$, we have $\delta^0(x,y) \ge 0$ on $S$, with equality at $(x_0,y_0)$.
On the other hand,  the symmetries of the cross-difference $\delta$ ensure that the Hessian
$h$ contributes the only non-vanishing term in the second order Taylor expansion of $\delta^0$:
more explicitly
\begin{eqnarray}\label{Taylor}
\nonumber
&\delta^0&(x_0+\Delta x,y_0 +\Delta y) \\
&=&
h((\Delta x, \Delta y),(\Delta x,\Delta y)) + o(|\Delta x|^2 +|\Delta y|^2) \\
&=& - \sum_{i=1}^{n^+}\sum_{j=1}^{n^-}
D^2_{x^iy^j} c(x_0,y_0) \Delta x^i \Delta y^j + o(|\Delta x|^2 +|\Delta y|^2)
\nonumber
\end{eqnarray}
as $(\Delta x,\Delta y)\to 0$.
It is then not so surprising to discover that the Hessian $h$ controls
the geometry and dimension of the support of any Kantorovich optimizer $\gamma$
near $(x_0,y_0)$ in various ways, as we now make precise
following Pass \cite{Pass12} and my joint works with Kim \cite{KimMcCann10}, Pass
and Warren \cite{McCannPassWarren12}.

Let $(k_+,k_0,k_-)$ be the $(+,0,-)$ signature of $h$, meaning
$k_+,k_0,k_- \in \N$ count the number of positive, zero, and negative eigenvalues of $h$
in one and hence any choice of coordinates.

\begin{claim}[Signature and rank]\label{C:signature and rank}
For each $(x_0,y_0) \in N$, the signature of \eqref{h=Hessian} is given by
$(k_+,k_0,k_-) = (k,n_+ + n_- - 2k, k)$ where $k \le \min\{n^+,n^-\}$ and $n^\pm = \dim M^\pm$.
We may henceforth refer to the integer $2k$ as the {\em rank} of $c \in C^2$ (and of $h$)
at $(x_0,y_0)$; it depends lower-semicontinuously on $(x_0,y_0)$.
\end{claim}

\begin{proof}
The sum $k_+ + k_0 + k_- = n_+ + n_-$ must agree with
the total dimension of $N=M^+ \times M^-$.  Since any
perturbation direction $(\Delta x, \Delta y)$ in which $\delta^0$ grows,  corresponds
to another direction $(-\Delta x, \Delta y)$ in which $\delta^0$ shrinks \eqref{Taylor},
it follows that $k_+=k_-$.
Thus $(k_+,k_0,k_-) = (k,n_+ + n_- - 2k, k)$.

In fact, since the matrix $h$ is symmetric, in any coordinate system we can find a
basis of orthogonal eigenvectors for $h$.  The preceding argument shows that if
$(\Delta x,\Delta y)$ is an eigenvector with eigenvalue $\lambda>0$ then
$(-\Delta x,\Delta y)$ is an eigenvector with eigenvalue
$-\lambda$.  In this case $\Delta x=- \frac{1}{2\lambda}\sum_{j=1}^{n^-} D^2_{xy^j} c(x_0,y_0) \Delta y^j$
is determined by $\Delta y$ and vise versa,
so at most $k \le \min\{n^+,n^-\}$ eigenvectors can correspond to
positive eigenvalues \cite{Pass12}.

Lowersemicontinuity of $k=k(x_0,y_0)$ follows from the fact that $c \in C^2$.
\end{proof}

The Hessian $h$ of the cross-difference also determines the {\em spacelike,
timelike,} and {\em lightlike}
cones $\Sigma^+, \Sigma^-$ and $\Sigma^0 \subset T_{(x_0,y_0)}N$
according to the definitions $\Sigma^\pm = \{V \in T_{(x_0,y_0)} N\mid \pm h(V,V) \ge 0\}$
and $\Sigma^0 = \Sigma^+ \cap \Sigma^-$.

\begin{definition}[Spacelike, timelike, lightlike]
A subset $S \subset N$ is {\em spacelike} if each (not necessarily continuous)
curve $t\in[-1,1] \longmapsto z(t) \in S$ differentiable at $t=0$
satisfies $h(\dot z(0),\dot z(0))\ge 0$, where $\dot z(0)$ is the tangent vector
and $h$ denotes the Hessian
\eqref{h=Hessian}  at $(x_0,y_0)=z(0)$.
  Similarly, $S$ is timelike (or lightlike)
if the inequality is reversed (or if both inequalities hold).
\end{definition}

Since we want to allow sets $S$ which are rough and potentially incomplete,
it is important to permit curves in the definition above whose continuity at $t=0$
may not extend to any neighbourhood of $t=0$.

\begin{lemma}[$c$-monotone implies spacelike]
Any $c$-monotone set $S \subset N$ is spacelike. 
\end{lemma}

\begin{proof}
Take any curve $(x(t),y(t)) \in S$, not necessarily continuous but differentiable at
$(x_0,y_0)=(x(0),y(0))$, with tangent vector $V=(\dot x(0),\dot y(0))$.
Since $\delta^0(x(t),y(t)) \ge 0$, from \eqref{h=Hessian}--\eqref{Taylor}
we see $h(V,V) \ge 0$, as desired.
\end{proof}

\begin{corollary}[Dimensional bounds]

If $c$ is $C^2$ and has rank $2k$ at a point $(x_0,y_0)$ where
$S$ has a well-defined tangent space $T$,  then
$c$-monotonicity of $S \subset N$ implies the dimension of this
tangent space satisfies
$\dim T \le n_+ + n_- - k$.
\end{corollary}

\begin{proof}
Fix coordinates on $N$.
As a consequence of the (Courant-Fischer) min-max formula for eigenvalues of $h$
at $(x_0,y_0)$, the signature $(k_+,k_0,k_-)=(k,n_++n_--2k,k)$ of $h$ limits
the maximal number of linearly independent tangent vectors to $N$
which are not timelike to $k_+ + k_0 = n_+ + n_- - k$.
Since the preceding lemma shows the tangent space $T$ of $S$
to be spanned by such a set of vectors,  its dimension satisfies the asserted bound.
\end{proof}

The following much stronger result of Pass \cite{Pass12} asserts $S$ is
contained in a spacelike Lipschitz submanifold of the prescribed dimension ---
hence implies differentiability a.e.\ as a consequence instead of a hypothesis.
The case $k=n_+=n_-$ was worked out earlier by
McCann, Pass and Warren \cite{McCannPassWarren12},
by adapting an idea of Minty \cite{Minty62} \cite{AlbertiAmbrosio99}
from the special case $c(x,y) = - x\cdot y$.

\begin{theorem}[Rectifiability \cite{Pass12}]\label{T:rectifiability}
If $c$ has rank $2k$ at $(x_0,y_0)$ and is $C^2$ nearby,
then on a (possibly smaller) neighbourhood $N_0 \subset M^+ \times M^-$ of $(x_0,y_0)$,
$c$-monotonicity of $S \subset N_0$ implies $S \subset L$ where $L\subset N_0$ is a
spacelike Lipschitz submanifold of dimension $\dim L \le n_+ + n_- - k$ with
$n_\pm = \dim M^\pm$.
\end{theorem}

\begin{proof}[Idea of proof]
A kernel of the proof can be apprehended already in the one-dimensional case
$n^\pm=1$.  When $c$ has rank zero, taking $L=N_0$ implies the result,  so assume
$c$ has full rank ($2k=2$), meaning
either $\partial^2 c/\partial x \partial y <0$ or $\partial^2 c/\partial x \partial y > 0$
near $(x_0,y_0)$.  In the first case,  $c$-monotonicity of $S$ implies
$S \cap R$  is contained in a non-decreasing subset
of any sufficiently small two-dimensional
rectangle $R=B_\epsilon(x_0)\times B_\epsilon(y_0)$.
This monotonicity is well-known in both mathematical \cite{Lorentz53}
and economic contexts \cite{Spence73} \cite{Mirrlees71}.
Rotating coordinates by setting $u=(x+y)/\sqrt{2}$ and $v=(y-x)/\sqrt{2}$,
the monotonicity is equivalent to asserting that $S$ is contained in
the graph of $\{(u,V(u))\}$ of a function $v=V(u)$ with Lipschitz constant one.
In the second case,  $c$-monotonicity would imply $S \cap R$ is non-increasing,
hence contained in a $1$-Lipschitz graph of $u$ over $v$.

The same argument carries over immediately to the bilinear cost $c(x,y) = - x \cdot y$
in higher dimensions $n^+=n^-$ \cite{Minty62}. 
For other costs with rank $2k=2n^+=2n^-$,  one can make a similar argument after
a linear change of coordinates $\tilde x = x-x_0$ and $\tilde y =\Lambda(y-y_0)$ chosen
so that in the new coordinates the cost takes the form
$\tilde c(\tilde x,\tilde y) = -  \tilde x\cdot \tilde y + o(\Delta \tilde x^2 + \Delta \tilde y^2)
+f(\tilde x) + f(\tilde y)$ \cite{McCannPassWarren12}.
The cases $k< \min\{n^+,n^-\}$ and $n^+ \ne n^-$ are worked out in \cite{Pass12}.
\end{proof}

When the rank of $c$ is maximal (i.e.\ $k = \min\{n^+,n^-\}$),  then the dimensional bound
is $\dim L \le \max\{n^+,n^-\}$.  Taking $n^+ \le n^-$ without loss of generality,
if the measure $\mu^-$ fills $M^-$ (say,  by being mutually absolutely continuous with
respect to Lebesgue measure in any coordinate patch),  the dimension of the Lipschitz
submanifold $L$ on which $\gamma$ concentrates cannot be less than $n^-$,  in which case
we see the bound given by the theorem is sharp: $\dim L = n^-$.

\begin{example}[Submodular costs on the line] \label{E:line}
If $M^\pm = \R$   there is a unique measure in $\Gamma(\mu^+,\mu^-)$ whose support
$S = \spt \gamma$
forms a non-decreasing subset of the plane.  This measure is the unique minimizer of
Kantorovich's problem \eqref{K min = M inf} for each cost $c \in C^1(\R^2)$ satisfying
$\partial^2 c /\partial x \partial y <0$; see e.g.\ \cite{McCann99}.  Apart from at most
countably many vertical segments,  the set $S$ is contained in the graph of some
$G:\R \longrightarrow \R \cup \{\pm \infty\}$ non-decreasing.  Unless $\mu^+$ has atoms,
the vertical segments
in $S$ are $\gamma$ negligible, in which case $\gamma = (id \times G)_\# \mu^+$ and Monge's
infimum is attained uniquely by~$G$.
\end{example}

\begin{example}[Transporting mass between spheres]\label{E:spheres}
Transporting mass on the surface of the earth has lead to consideration of the
cost function $c(x,y) = \frac{1}{2}|x - y|^2$ restricted to the boundary of the unit
sphere $x,y \in \partial B^{n+1}_1(\0) \subset \R^{n+1}$ so that $0 \le c \le 2$
\cite{ChiapporiMcCannNesheim10}\cite{AhmadKimMcCann11}\cite{FigalliKimMcCann-econ}\cite{McCannSosio11},
a problem considered earlier in the context of shape recognition \cite{GangboMcCann00}\cite{Ahmad04}.
The restricted cost has rank $2n$ except on the degenerate set
$c=1$,  where it has rank $2n-2$.  Thus any $c$-cyclically monotone subset $S$ of
the $2n$-dimensional product space $\partial B^{n+1}_1(\0) \times \partial B^{n+1}_1(\0)$
has dimension at most $n$ except along the degenerate set, where it 
has dimension at most $n+1$ (in spite of the fact that the degenerate set is $2n-1$ dimensional).
Since the degenerate set separates the orientation preserving and orientation reversing
parts $S^+$ and $S^-$ of $S$,  this means that
$S^+$ cannot intersect $S^-$ transversally (except in dimension $n=1$);
instead, if $S^+$ meets $S^-$
at a point where both have $n$-dimensional tangent spaces, these spaces must
have $n-1$ directions in common.  For example,
if $n=3$, and both $S^+$ and $S^-$ are generically $3$-dimensional,  but their union is
contained in a $4$-dimensional Lipschitz submanifold, whereas the cost degenerates
on a smooth $5$-dimensional hypersurface.
\end{example}

In summary,  $c$-monotonicity implies rectifiability of
$S=\spt\gamma \subset N=M^+ \times M^-$ in a dimension
determined locally by the rank of the Hessian $h$ of the cross-difference
$\delta^0 \in C^2$;
moreover $S$ must be spacelike with respect this Hessian \eqref{h=Hessian}.
If $h$ is non-degenerate,  we will eventually see that $h$ can be viewed as a
pseudo-metric on $N$ whose Riemannian sectional curvatures combine with $\mu^\pm$
to determine smoothness of $S$.

\section{When do optimal maps exist?}

We now turn to the more classical question of attainment of the infimum \eqref{K min = M inf}.
To expect existence of Monge maps, we generally need $\mu^+$ to be more than atom-free.
We need $\mu^+$ not to concentrate positive mass on any lower dimensional submanifold
of $\X$, or more precisely on any hypersurface parameterized locally in coordinates
as the graph of a difference of convex functions.
This condition, proposed by Gangbo and McCann \cite{GangboMcCann96},
is sharp in a sense made precise by Gigli \cite{Gigli11},
and implies Lipschitz continuity and
$C^2$-rectifiability of the hypersurfaces in question.  Absolute continuity of $\mu^+$
in coordinates --- i.e.\ the existence of a density $f^+$ such that $d\mu^+(x)=f^+(x)dx$ ---
is more than enough to guarantee this.  However,  we also require further structure
of the transportation cost.

For $c \in C^1(N)$,
the Gangbo \cite{Gangbo95} and Levin \cite{Levin99} criterion for
existence of Monge solutions $G:\X \longrightarrow \Y$ given in Appendix \ref{A:MTW}
is equivalent to: \\
{\em
\Aone$_+$ For each $x_0 \in M^+$ and $y_0 \ne y_1 \in M^-$, assume
$x \in M^+ \longmapsto \delta^0(x,y_1)$ has no critical points,
where $\delta^0(x,y_1)=\delta(x,y_1;x_0,y_0)$ is from \eqref{cross-difference}.} \\
Naturally,  this implies $n^+ \ge n^-$,  due to the fact we cannot generally hope to
use a (rectifiable) map $G$ on a low dimensional space to spread a measure over a higher
dimensional space.
In fact, \Aone$_+$ implies something stronger:  namely that every solution of the Kantorovich problem
is a Monge solution.  This in turn implies uniqueness of the Kantorovich (and hence
Monge) solution,  for the following reason.  Suppose two Kantorovich solutions exist,
and both correspond to Monge solutions: $\gamma_0 = (id \times G_0)_\# \mu^+$ and
$\gamma_1 = (id \times G_1)_\# \mu^+$.  Linearity of the Kantorovich problem implies
$\gamma_2 :=(\gamma_0 + \gamma_1)/2$ is again a solution,  hence by \Aone$_+$ must
concentrate on the graph of a map $G:\X \longrightarrow \Y$.  It is then easy to argue
$\gamma_i=(id \times G)_\#\mu^+$ for $i=0,1,2$ as in e.g.\ \cite{AhmadKimMcCann11}.
This implies $\gamma_0=\gamma_1$; moreover $G_0=G=G_1$ $\mu$-a.e.  Thus we arrive at
the following theorem \cite{Gangbo95} \cite{Levin99} \cite{AhmadKimMcCann11} \cite{Gigli11}:

\begin{theorem}[Existence and uniqueness of optimal maps]\label{T:LevinGigli}
Let $\mu^\pm$ be probability measures on manifolds $M^\pm$,  with a cost
$c \in C^1(M^+ \times M^-)$ which is bounded and satisfies \Aone$_+$.
If $\mu^+$ assigns zero mass to each Lipschitz hypersurface in $M^+$,  then Kantorovich's minimum
is uniquely attained,  and the minimizer $\gamma = (id \times G)_\# \mu^+$ vanishes
outside the graph of a map $G$ solving Monge's problem.  In fact, not all Lipschitz
hypersurfaces are required: it is enough that
$\mu^+$ vanish on each hypersurface locally parameterizable in coordinates as the graph
of a difference of two convex functions.
\end{theorem}

Notice  \Aone$_+$ asserts the restriction of $\delta^0$ to each horizontal fibre
$\X \times \{y_1\}$ has no critical points,
except on the fibre $y_1=y_0$  where
$\delta^0$ vanishes identically.  To guarantee invertibility of the map $G$,
we need the same condition to hold for the reflected cost $c^*(y,x) := c(x,y)$,  meaning the
roles of $\X$ and $\Y$ are interchanged.  If both $c$ and $c^*$ satisfy \Aone$_+$,
we say \Aone\ holds.  Thus \Aone\ is equivalent to asserting that $(x_0,y_0)$ is
the only critical point of $\delta^0(x,y)$.


Many interesting costs,  such as $c(x,y)=h(x-y)$ with $h$ strictly convex or concave
on $M^\pm =\R^n$ satisfy these hypotheses.  The most classical of these is the
Euclidean distance squared \cite{Brenier87} \cite{CullenPurser84}
\cite{SmithKnott87} \cite{RuschendorfRachev90}.
Regularity of the convex gradient map it induces, generalizing Example \ref{E:line}, was
established by Delano\"e \cite{Delanoe91} for $n=2$
and Caffarelli \cite{Caffarelli91} \cite{Caffarelli96b}
and Urbas \cite{Urbas97} for $n \ge 3$.

\begin{example}[Euclidean distance squared]\label{E:Brenier's cost}
If $M^\pm \subset \R^n$ and $\mu^+$ vanishes on all hypersurfaces,  there
is a unique measure in $\Gamma(\mu^+,\mu^-)$ concentrated on the graph of the
gradient of a convex function $u:\R^n \longrightarrow \R \cup \{+\infty\}$.
This measure is the unique minimizer of
Kantorovich's problem \eqref{K min = M inf} for the cost $c(x,y) = \frac{1}{2}|x-y|^2$
\cite{Brenier91} \cite{McCann95}.  If $d\mu^\pm = f^\pm d\Hn$ are both absolutely continuous
with respect to Lesbesgue,  the Monge-Amp\`ere equation $f^+(x) = f^-(Du(x))\det D^2 u(x)$
holds $\mu^+$-a.e.\ \cite{McCann97}.  If moreover, $\log f^\pm \in L^\infty(M^\pm)$
with $M^-$ convex and $\Hn(\partial M^+)=0$,  then $u \in C^{1,\alpha}_{loc}(M^+)$ for
some $\alpha>0$ \cite{Caffarelli91} estimated in \cite{ForzaniMaldonado04}.
If, in addition $f^\pm \in C^{1,\bar \beta}$ and
$M^+$ and $M^-$ are both smooth and strongly convex --- meaning the principle curvatures
of their boundaries are all strictly positive --- then $u \in C^{2,\beta}(\bar M^+)$ for all
$0<\beta<\bar \beta<1$ \cite{Delanoe91} \cite{Caffarelli96b} \cite{Urbas97}.
Higher regularity follows from smoothness of $f^\pm$.
\end{example}

On the other hand,  \Aone$_+$ also fails for many
interesting geometries.  We mention two such examples.  In the first ---
the cost function of interest to Monge \cite{Monge81} --- optimal maps turn out
to exist but are not unique.  Their non-uniqueness was quantified with
Feldman \cite{FeldmanMcCann02u}.
In the second,  Monge's infimum turns out not to be attained,
despite the fact that the Kantorovich minimizer is unique.

\begin{example}[Uniqueness fails for Monge's cost]
Let open sets $M^\pm \subset \R^n$ have finite volume and $c(x,y) = |x-y|$.
Monge was originally interested in transporting the uniform measure
$\mu^\pm = \frac{1}{\Hn(M^\pm)}\Hn$ from one domain to the other, when $n=3$
and $\Hn$ denotes the $n$-dimensional Hausdorff measure, and coincides with
Lebesgue measure in this case \cite{Monge81}.
Taking $M^+$ disjoint from $M^-$ ensures smoothness of $c$.  Notice that when $n=1$ and
$M^+$ and $M^-$ are disjoint intervals,  every $\gamma \in \Gamma(\mu^+,\mu^-)$ has the
same total cost $cost(\gamma)$.  In this case the solution to Kantorovich's problem is badly
non-unique.  Clearly \Aone$_+$ also fails in this case.  In higher dimensions,  the situation
is slightly less degenerate since the cost takes a range of values on $\Gamma(\mu^+,\mu^-)$,
but it remains true that its extrema are not uniquely attained.  In this setting,  it can be
a difficult problem to show that Monge's infimum is attained. This problem was  first solved by
Sudakov in the plane $n=2$; he asserted a result in all dimensions but it was later discovered
that one of his claims sometimes fails if $n>2$.  This existence result was extended to higher
dimensions by Evans and Gangbo, assuming $\mu^\pm$ to be given by
Lipschitz continuous densities on $\R^n$ \cite{EvansGangbo99}, and for general
absolutely continuous densities $\mu^\pm$
by Ambrosio \cite{Ambrosio03},  Trudinger-Wang \cite{TrudingerWang01}
and Caffarelli-Feldman-McCann \cite{CaffarelliFeldmanMcCann00} simultaneously and
independently.
The last group also considered costs given by non-Euclidean norms,
but with smooth and strongly convex unit balls,
restrictions removed in a seqeunce of papers by different teams of authors
including Ambrosio, Bernard, Buffoni, Bianchini, Caravenna, Kirchheim, and Pratelli,
and culminating in work of Champion and DePascale \cite{ChampionDePascale11}.
%
\end{example}

On the other hand, if $M^+$ is a compact manifold without boundary,
it is evident that $x \in M^+ \longmapsto \delta^0(x,y_1)$ must attain at
least one maximum and one minimum so that --- as long as the cost is assumed
differentiable --- it is clear that \Aone$_+$ cannot be satisfied.
In this case,  it will not always be true that Monge's
infimum \eqref{K min = M inf} is attained, as my examples
with Gangbo \cite{GangboMcCann00} show:

\begin{example}[Transporting mass between spheres, revisited]\label{E:spheres revisited}
Restrict $c(x,y) = \frac{1}{2}|x-y|^2$ to
$M^\pm = \partial B_1(\0) \subset \R^{n+1}$ so that $0\le c \le 2$,
as in Example \ref{E:spheres}.  Take $\mu^\pm$ to be mutually
absolutely continuous with respect to surface area $\Hn$ on their
respective spheres,  but take most of the mass of $\mu^+$ to be concentrated
near the north pole and most of the mass of $\mu^-$ to be concentrated near
the south pole.  Then Monge's infimum \eqref{K min = M inf} will not be attained,
despite the fact that the Kantorovich minimizer $\gamma$ is unique.
The intersection of $S=\spt \gamma$ with the set $c \le 1$ is contained in the
graph of a map $G:M^+ \longrightarrow M^-$,  while the intersection $S \cap \{c \ge 1\}$
is contained in the graph of a map $H:M^- \longrightarrow M^+$ --- sometimes called an
{\em antigraph}.  If the densities $f^\pm = d\mu^\pm / d\Hn$ are both bounded,
so that $\log f^\pm \in L^\infty$,  then $G$ is a homeomorphism of
$\partial B_1$ and $H$ may be taken to be continuous \cite{GangboMcCann00};
both maps enjoy a local H\"older exponent of continuity $\alpha=1/(4n-1)$
except possibly where their graphs touch the set $\{c=1\}$ where the rank of
$c$ drops from $2n$ to $2n-2$ \cite{McCannSosio11}.
It may be possible to improve this H\"older exponent to $\alpha=1/(2n-1)$ using techniques
of Liu \cite{Liu09},  but even when $f^\pm$ are smooth we have no idea how to prove $G$
will be smoother, nor how to extend H\"older continuity of $G$ up to the degenerate set
$\{c=1\}$.
\end{example}

Notice that global differentiability of the cost is crucial to this discussion.
For costs whose differentiability fails --- even on a small set such as the Riemannian
cut locus --- the theorem which follows gives many natural examples where existence and
uniqueness both hold.

\begin{theorem}[Minimizing Riemannian distance squared]\label{T:Riemannian}
Let $c(x,y) = d^2(x,y)/2$ be the square distance induced by some Riemannian
metric on a compact manifold $M^+=M^-$.  If $\mu^+$ is absolutely continuous
(with respect to Riemannian volume) then the Kantorovich minimizer is unique in
\eqref{K min = M inf}, and takes the form $\gamma = (id \times G)_\# \mu^+$
for a map solving Monge's problem \cite{McCann01}.
In case $M^\pm$ are round spheres \cite{Loeper08p}
(or quotients \cite{DelanoeGe10}, submersions \cite{KimMcCann08p} or
products thereof \cite{FigalliKimMcCann-spheres}),  and both $\mu^\pm$ are given
by smooth positive densities with respect to surface area,  then the map $G$ will
be a smooth diffeomorphism.
\end{theorem}

Notice that the existence and uniqueness asserted in Theorem \ref{T:Riemannian}
is not quite a corollary
of Theorem \ref{T:LevinGigli},  since compactness
of the manifold $M^\pm$ forces the cut-locus to be non-trivial. Here the cut-locus
is defined as (the closure of) the set of points where differentiability
of the cost $c=d^2/2$ fails.

\section{When are optimal measures unique?}

The preceding section shows that if the cross-difference $\delta^0(x,y)=\delta(x,y;x_0,y_0)$
has no critical points unless $x=x_0$ or $y=y_0$,  then Monge's problem is
soluble and the Kantorovich problem admits a unique solution.  Although very
useful when it applies,  this criterion is not satisfied in all cases of interest.
--- for example,  when trying to minimize the restriction of the quadratic cost
$c(x,y) = |x-y|^2/2$ to the Euclidean unit sphere $M^\pm = \partial B_1(\0)\subset \R^{n+1}$.
In such situations,  my results with Chiappori, Nesheim \cite{ChiapporiMcCannNesheim10},
Ahmad and Kim \cite{AhmadKimMcCann11} may be useful:

\begin{theorem}[Uniqueness of minimizer for subtwisted costs]
\label{T:subtwist}
Fix Borel probability measures $\mu^\pm$ on manifolds $M^\pm$,
with $\mu^+$ vanishing on each hypersurface in $M^+$, and a
bounded cost $c \in C^1(M^+ \times M^-)$.  Suppose for each
$x_0 \in M^+$ and $y_0 \ne y_1 \in M^-$,  the cross-difference
$\delta^0(x,y) := \delta(x,y; x_0,y_0)$ from \eqref{cross-difference}
satisfies
\begin{equation}\label{subtwist}
x \in M^+ \longmapsto \delta^0(x,y_1)
\begin{array}{l}
\mbox{\rm has at most two critical points, namely, a unique} \cr
\mbox{\rm global minimum and a unique global maximum.}
\end{array}
\end{equation}
Then the Kantorovich problem has a unique solution,  and it takes the form
$\gamma = (id \times G)_\#\mu + (H \times id)_\# (\mu^--G_\#\mu)$ for some
maps $G:M^+ \longrightarrow M^-$ and $H:M^- \longrightarrow M^+$ and
non-negative measure $\mu \le \mu^+$ such that $\mu^--G_\#\mu$ vanishes
on the range of $G$.
\end{theorem}

In other words,  the unique Kantorovich solution concentrates on the union of
the graph and an antigraph, 
of $G:M^+ \longrightarrow M^-$ and of 
$H:M^- \longrightarrow M^+$ respectively.  Notice that if the manifold $M^+$ is compact,
hypothesis \eqref{subtwist} restricts its Morse structure to be that of the sphere,
so the theorem generalizes of Example \ref{E:spheres revisited}:
However apart from the continuity results of \cite{GangboMcCann00} \cite{McCannSosio11}
and \cite{KitagawaWarren11p},
it is not known when $G$ and $H$ can be expected to be smooth.  It is even more
shocking that no criterion analogous to Theorem \ref{T:subtwist} is known which guarantees
uniqueness of Kantorovich minimizer on the torus --- or indeed on any other compact manifolds
$M^\pm$ apart from the sphere.


\section{When are optimal maps continuous? Smooth?}

Examples \ref{E:Brenier's cost}, \ref{E:spheres revisited} and Theorem \ref{T:Riemannian}
complement Theorems \ref{T:rectifiability} and \ref{T:LevinGigli} by providing a variety
of settings where the optimal map $G$ is continuous and/or
support of the optimal measure can actually be shown to be smooth.
In each case,  we need the cost to be suitable,  the domain geometry
to be favorable,  and the measures to be positive, bounded and possibly smooth.

Following the analysis of a number of such examples, including the restriction of
$c(x,y) = -\log|x-y|$ to the unit sphere \cite{Wang96} \cite{Wang04},
a general theory for addressing such questions has begun to be developed, starting from
the pioneering work of Ma, Trudinger and Wang \cite{MaTrudingerWang05}, who identified
conditions 
on the transportation cost $c$ which are close to being necessary and sufficient for
smoothness of $G$. 
Their work is set on bounded domains $M^\pm \subset \R^{n}$, and
as we now explain,  each of their conditions 
can be reformulated in terms of the topology
and geometry of the cross-difference
$\delta^0(x,y) = \delta(x,y; x_0,y_0)$ from \eqref{cross-difference}
and its Hessian $h = \frac{1}{2}{\rm Hess}_{(x_0,y_0)} \delta^0$. 

Where $c$ has full rank $2n$,  the Hessian $h$ is non-degenerate and can be understood
as a pseudo-Riemannian metric tensor on the product space.  According to
Claim \ref{C:signature and rank},  this pseudo-metric tensor is not positive
definite,  but instead has the same number of spacelike and timelike dimensions.
At each point point $(x_0,y_0) \in N$,  the light-cone separating these spacelike
from timelike directions consists of the tangent spaces to
$\{x_0\} \times M^-$ and $M^+ \times \{y_0\}$.  However,  just as in Riemannian
(and Lorentzian) geometry,  the pseudo-metric tensor $h$ induces a geometry
on the product space $N = M^+ \times M^-$,  including geodesics and a pseudo-Riemannian
curvature tensor $R_{i'j'k'l'}$, which assigns sectional curvature
$$
\sec^{(\bar N,h)}_{(x_0,y_0)} P\wedge Q = \sum_{1 \le i',j',k',l' \le 2n} R_{i'j'k'l'} P^{i'}Q^{j'}P^{k'}Q^{l'}
$$
to each pair of vectors $P,Q \in T_{{x_0,y_0}}N$.  The explicit formulae expressing geodesics
and the curvature tensor \eqref{Riemann tensor} in terms of $h$ can be found
in \cite{KimMcCann10} or deduced from Appendix \ref{A:MTW};  they are
precisely analogous to the Riemannian case.

In terms these notions,  we may now state conditions equivalent to those of
Ma, Trudinger and Wang \Aone$'$--\Afour$'$ found in Appendix \ref{A:MTW} below:
\smallskip

\Azero\  $c \in C^4(\bar N)$, and for each
$(x_0,y_0) \in \bar N = \bar M^+ \times \bar M^- \subset \R^{n} \times \R^n$:
\smallskip

\Aone\  $(x,y) \in \bar N \mapsto \delta^0(x,y)$ 
from \eqref{cross-difference}
has no critical points save
$(x_0,y_0)$;
\smallskip

\Atwo\ $c$ has rank $2n$, so $h=$Hess$_{(x_0,y_0)} \delta^0$
defines a pseudo-metric tensor;
\smallskip

\Athree\ $\sec^{(\bar N,h)}_{(x_0,y_0)} (p \oplus \0)\wedge (\0 \oplus q) \ge 0$ for each lightlike
$(p,q) \in T_{(x_0,y_0)}\bar N$; 
\smallskip

\Afour\ the sets $\{x_0\} \times M^-$ and $M^+ \times \{y_0\}$ are $h$-geodesically convex.
\smallskip \\
Here a subset $Z \subset \bar N$ is said to be {\em $h$-geodesically convex}
if each pair of points
$(x_0,y_0)$ and $(x_1,y_1) \in Z$ can be joined by an geodesic in $\bar N$ lying entirely within $Z$,
geodesics being defined relative to the pseudo-metric $h$.


The most intriguing of these conditions is the curvature condition \Athree.
A large body of example costs which satisfy
\cite{MaTrudingerWang05} \cite{Loeper08p}
\cite{DelanoeGe10}
\cite{KimMcCann10}
\cite{FigalliRifford09}
\cite{FigalliRiffordVillani12}
\cite{LeeMcCann11} \cite{Lee10} \cite{LeeLi09p}
\cite{DelanoeRouviere12p}
or violate it \cite{MaTrudingerWang05} \cite{Loeper09} have
now been established.  Among the former we may mention the restriction of the Euclidean
distance squared to the graphs $M^\pm \subset \R^{n+1}$ of any pair of
$1$-Lipschitz convex functions \cite{MaTrudingerWang05},
as well as the Riemannian distance squared on
the round sphere \cite{Loeper08p},
and any products \cite{KimMcCann10}, submersions \cite{KimMcCann10} or perturbations
\cite{DelanoeGe10} \cite{FigalliRifford09} \cite{FigalliRiffordVillani12} thereof.
Among the latter we may mention
the Riemannian distance squared on any manifold $(M, g_{ij})$ with a non-negative sectional
curvature somewhere \cite{Loeper09},  and the restriction of the Euclidean
distance squared to the graphs of two functions in $\R^{n+1}$,  one of which is convex
and the other non-convex \cite{MaTrudingerWang05}.
Thus the distance squared in hyperbolic space $c=d_{\mathbf H^n}^2$
violates \Athree,  though $c= -\cosh d_{\mathbf H^n}$ satisfies it \cite{Li09} \cite{LeeLi09p}.

To conclude continuity or higher regularity of $G$ at present requires a slight strengthening
of one of the geometric conditions \Athree\ or \Afour.  If the inequality in \Athree\ holds
strictly whenever the $h$-orthogonal vectors $p \oplus \0$ and $\0 \oplus q$ are
non-vanishing,  we denote that by \Athrees.  If instead the geodesic convexity of the sets
in \Afour\ is strong (i.e.\ $2$-uniform,  in the sense of Example \ref{E:Brenier's cost}
or Appendix \ref{A:MTW}), we denote that by \Afours.  Under these assumptions the following
extensions of Theorem \ref{T:LevinGigli} and Example \ref{E:Brenier's cost}
have been proved,  in works of Ma, Trudinger, Wang, Loeper, Liu, Figalli, Kim and myself.

\begin{theorem}[Continuity and smoothness of optimal maps]
Assume \Azero-\Afour\ hold,  and $d\mu^\pm = f^\pm d\Hn$ are given by densities
satisfying $\log f^\pm \in L^\infty (U^\pm)$ with $U^-=M^-$ and $U^+ \subset M^+$ open.
(a) If \Athrees\ holds,
the map $G \in C^\alpha_{loc}(U^+,M^-)$ is H\"older continuous \cite{Loeper09},
with an exponent $\alpha = 1/(2n-1)$ known to be sharp \cite{Liu09}.
(b) If \Athrees\ fails but \Afours\ holds,  the same conclusion persists but with an unknown
exponent $\alpha$ independent of $c$, but
presumed to depend on $\|\log (f^+/f^-) \|_{L^\infty(U^+ \times M^-)}$.
Either way, higher interior regularity of $G$ follows from smoothness of
$f^\pm$ \cite{LiuTrudingerWang10}.
If, $U^\pm = M^\pm$ and $f^\pm$ are smooth in case (b),  the smoothness of $G$ shown
in \cite{MaTrudingerWang05}
extends up to the boundary~\cite{TrudingerWang09b}.
\end{theorem}

It is possible to construct smooth bounded $f^\pm$ for which
continuity of $G$ fails in the absence of either \Athree\ or \Afour\
as was done by Loeper \cite{Loeper09} and by
Ma, Trudinger and Wang \cite{MaTrudingerWang05} respectively.
Still,  there are few results quantifying the discontinuities of $G$, except for the cost
$c(x,y) = \frac{1}{2}|x-y|^2$ of Example \ref{E:Brenier's cost} \cite{Yu07}
\cite{Figalli10} \cite{FigalliKim10},
for which examples of discontinuous maps go back to Caffarelli \cite{Caffarelli91}.

\section{Closed forms and $c$-cyclical monotonicity}

The sections above have discussed many necessary conditions
for optimality of $\gamma$,  but few sufficient conditions.
In fact,  for bounded continuous $c\in C(M^+ \times M^-)$,
a condition on the support $S=\spt \gamma$ well-known
to be necessary and sufficient for
optimality in $\Gamma(\pi^+_\#\gamma,\pi^-_\#\gamma)$ is given
by: 

\begin{definition}[$c$-cyclical monotonicity]
A set $S \subset M^+ \times M^-$ is {\em $c$-cyclically monotone} if and only if
each $k\in \N$, sequence $(x_1,y_1), \ldots, (x_k,y_k) \in S$, and permutation $\tau$
on $k$ letters satisfy the following inequality:
\begin{equation}\label{c-cyclical monotonicity}
\sum_{i=1}^k c(x_i,y_i) \le \sum_{i=1}^k c(x_{\tau(i)},y_i).
\end{equation}
\end{definition}

This result can be found in Pratelli \cite{Pratelli08} or
Schachermayer-Teichmann \cite{SchachermayerTeichmann09},
building on earlier works of Knott-Smith, Gangbo-McCann,
R\"uschendorf, and Ambrosio-Pratelli.
The case $k=2$ corresponds to the $c$-monotonicity condition which implies that $S$
is $h$-spacelike.  The result quoted above shows
the cross-difference $\delta(x,y; x_0,y_0)$ is just the first in an infinite sequence
of functions whose non-negativity on $S^k$ for each $k \in \N$
characterizes optimality of $\gamma$.  In fact,  since all permutations are
made up of cycles,  for each $k$ it is enough to check \eqref{c-cyclical monotonicity}
for the cyclic permutation $\tau(i)=i+1$ if $i<k$ with $\tau(k)=1$.  This family of
conditions has a differential topological content whose relevance we now try to make clear.

Choose any map $G: U^+ \subset M^+ \longrightarrow M^-$ defined on a subset $U^+ \subset M^-$,
whose graph lies inside $S$.  Any differentiable
loop $\sigma:S^1 \longrightarrow M^+$ may be approximated
by $x_i = \sigma(\theta_i)$ for a partition $0<\theta_1 < \cdots < \theta_k \le 2\pi$ as
fine as we please.  The non-negative sums \eqref{c-cyclical monotonicity}
then approximate Riemann sums for the integral
$$
0 \le \int_0^{2\pi} D_x c(\sigma(\theta),G(\sigma(\theta))) \cdot \sigma'(\theta) d\theta
$$
arbitrarily closely.  If the form $x \in U^+ \longmapsto D_x c(x,G(x))$ is continuous on an open
set $U^+\subset M^+$ containing $\sigma$, then the Riemann integral exists.
Since the curve can be traversed
in either direction,  the non-negative integral must actually vanish,  hence the form must be closed:
for $U^+$ simply connected, there would exist $u \in C^1_{loc}(U^+)$ such that $D_x c(x,G(x)) = Du(x)$.
Similarly,  if $G$ could be continuously inverted on a simply connected domain $U^-\subset M^-$,
there would exist $v \in C^1_{loc}(U^-)$ such that $D_y c(G^{-1}(y),y) = Dv(y)$.
These suppositions are not so implausible when \Aone--\Atwo\ hold,  since
$S$ at least coincides with the graph of a map $G$ and has a well-defined tangent space
$\Hn$-almost everywhere.

However, despite the fact that neither $G$ nor its inverse will be continuous in general,
some vestige of this integrability persists.  If $c$ is Lipschitz continuous for example,
then \eqref{c-cyclical monotonicity} implies the existence of Lipschitz $u,v$ such that
$c(x,y)-u(x) - v(y) \ge 0$ on $N=M^+ \times M^-$ with equality holding throughout $S$.
This fact,  which goes back to \cite{Rockafellar66} \cite{Rochet87},
is in many senses better than mere integrability of a form:
it requires no topology restriction on the domains,  and not only do we get the first-order
condition $Du(x)=D_x c(x,y)$ for those points $(x,y) \in S$ with $x$ in the set of
$\Hn$ full measure $\dom Du$ where $u$ is differentiable; as a second-order condition
we get positive-definiteness of the matrix
$D^2_{xx} c (x,y)-D^2 u(x) \ge 0$ if $x \in \dom D^2 u$,  and analogous conditions for $v$.
Verily is $S$ contained in the gradient of a convex function when
$c(x,y) = - x \cdot y$ or $c(x,y) = \frac{1}{2}|x-y|^2$ on $U^\pm \subset \R^n$.

As Gangbo and McCann argue \cite{GangboMcCann96},
this rough integrability result of Rockafellar and Rochet
implies the famous duality of Kantorovich \cite{Kantorovich42},
Koopmans and Beckmann \cite{KoopmansBeckmann57}:
\begin{equation}\label{dual}
\min_{\Gamma(\mu^+,\mu^-)} \int_{M^+ \times M^-} c(x,y) d\gamma(x,y) =
\sup_{(u^+,u^-) \in Lip_c} \int_{M^+} u^+ d\mu^+ + \int_{M^-} u^-d\mu^-
\end{equation}
with the supremum over
\begin{equation}\label{Lipc}
Lip_c := \{ u^\pm \in L^1(d\mu^+) \mid c(x,y) \ge u^+(x) + u^-(y) {\rm\ throughout}\ N \}
\end{equation}
being attained at $(u^+,u^-)=(u,v)$.  Indeed, for any
$(u^+,u^-) \in Lip_c$,  integrating the inequality
\eqref{Lipc} against $\gamma \in \Gamma(\mu^+,\mu^-)$ yields
\begin{equation}\label{easy dual}
\int_{M^+ \times M^-} c d\gamma \ge \int_{M^+} u^+ d\mu^+ + \int_{M^-} u^- d\mu^-.
\end{equation}
Thus the $\min$ dominates the $\sup$ in \eqref{dual}.  Starting from
$\gamma \in \Gamma(\mu^+,\mu^-)$ with $c$-cyclically monotone support,
Rochet's generalization of Rockafellar's theorem provides $(u^+,u^-) =(u,v) \in Lip_c$
--- bounded and Lipschitz if $c$ is --- such that equality holds in
\eqref{easy dual}, and hence in \eqref{dual} as desired.

\section{Connections to differential geometry}

We have already seen that the pseudo-Riemannian geometry
induced on the product space $N=\X \times \Y$ by the metric tensor
$h=\frac{1}{2}\Hess \delta^0(x_0,y_0)$
plays a key role in determining whether or not maps $y=G(x)$ which solve Monge's problem
\eqref{Monge cost} are smooth.
Here $h$ is the Hessian of the cross-difference \eqref{cross-difference}--\eqref{h=Hessian}
associated to the cost $c$.  The antisymmetry
$$\delta(x,y;x_0,y_0) = \delta (x_0,y_0; x,y) = - \delta (x,y_0; x_0,y)
$$
ensures that $h$ vanishes on $n \times n$ diagonal blocks.  The involution
$U(\Delta x,\Delta y)=(\Delta x,-\Delta y)$ on $T_{(x_0,y_0)}N$ allows us to
define an antisymmetrized analog of $h$ by
$$\omega(P,Q)= h(P,U(Q)).$$
Here $\omega$ turns out to be a symplectic form if and only if $h$
has the full rank $2n=2n^\pm$ that we often assume.
Notice the similarity to K\"ahler geometry,  with the splitting
$T_{(x_0,y_0)}N = T_{x_0} \X \oplus T_{y_0} \Y$ of the tangent space
associated to $U$ playing the role of the almost complex structure $J$,
and the cost $c$ playing the role of the K\"ahler potential.  For geometric
measure theory in such geometries see Harvey and Lawson \cite{HarveyLawson10p}.

Kim and McCann showed
that any $c$-optimal diffeomorphism $G:\X \longrightarrow \Y$ has a graph which
is $\omega$-Lagrangian in addition to being $h$-spacelike.  Conversely,  when
\Azero--\Afour\ hold,  then any diffeomorphism with an $\omega$-Lagrangian and
$h$-spacelike graph is necessarily $c$-optimal \cite{KimMcCann10}.  Here a submanifold
$S \subset N$ is called $\omega$-Lagrangian if $\omega(P,Q)=0$ for every
pair of tangent vectors $P,Q \in T_{(x_0,y_0)}N$.  Being $\omega$-Lagrangian
is essentially the integrability condition which asserts closure of the form
$D_x c|_{(x,G(x))}$ on $M^+$;  it amounts to equality of the cross-derivatives
$\partial G^i/\partial x^j = \partial G^j/\partial x^i$ which imply the existence of $u$
such that $G(x) = Du(x)$ in case $c(x,y) = - x \cdot y$.

So far these geometric structures --- the pseudo-metric $h$, symplectic form $\omega$,
$c$-cyclical monotonicity, and $c$-optimality --- reflect only the cost function $c(x,y)$,
and not the densities $d\mu^\pm(x) = f^\pm(x) dx$.  Remarkably, however,  there is a
conformally equivalent pseudo-metric
$$\tilde h(x_0,y_0) = \left(\frac{f^+(x_0) f^-(y_0)}{|\det \partial^2 c/\partial x^i \partial y^j|} \right)^{1/n} h(x_0,y_0)
$$
for which the graph $\graph(G)$ of an optimal mapping $G_\# \mu^+ = \mu^-$ turns out to
be a zero mean curvature surface  ---
and in fact $\tilde h$-volume maximizing among homologous
surfaces. This surprising connection of optimal transportation to
geometric measure theory was discovered with Kim and Warren \cite{KimMcCannWarren10}.

Thus the properties of optimal maps relate to both sectional and mean curvatures with
respect to $\tilde h$.  On the other hand,  in the special case of the quadratic
cost $c = d^2$ on a Riemannian manifold $M=M^\pm$,
several surprising connections relate optimal
transportation to the Riemannian geometry of $(M,g_{ij})$.
For example,  in this case Loeper and Villani conjecture \cite{LoeperVillani10}
--- and in some cases have proved
--- \Athrees\ implies {\em convexity} of the tangent injectivity locus,  which is to say the cut
locus of each given point $x_0 \in M$, lifted to the tangent space $T_{x_0} \X$
by the Riemannian exponential $\exp_{x_0}^{-1}$.

An earlier development involved
lifting the metrical distance $d$ from $M$ to the space $P(M)$ of Borel probability measures
$\mu^\pm \in P(M)$ using the minimal transportation cost
$d_2(\mu^+,\mu^-)= \sqrt{cost(\gamma)}$ with respect to distance squared
$c=d^2$ \cite{BenamouBrenier00} \cite{Dudley76} \cite{Otto01}.
Geodesic convexity of various entropy
functionals on $P(M)$ turns out to be equivalent to Ricci non-negativity of $(M,g)$.
This was shown by von Renesse and Sturm \cite{vonRenesseSturm05},  building on work of myself \cite{McCann97},
Cordero-Erausquin, Schmuckenschl\"ager and
I \cite{CorderoMcCannSchmuckenschlager01}, and Otto and Villani \cite{OttoVillani00}.
This idea was turned on its head by Lott-Villani \cite{LottVillani09} and independently
Sturm \cite{Sturm06ab},  who used geodesic convexity of the same entropies
to {\em define} Ricci non-negativity in (not necessarily smooth) metric-measure spaces.
This non-negativity is stable under measured Gromov-Hausdorff convergence, and
has significant consequences.

\appendix

\section{Ma-Trudinger-Wang conditions}\label{A:MTW}

The conditions \Azero-\Afour\ above have been synthesized
in a language
selected to manifest their topological and geometric invariance --- aspects not
readily apparent \cite{Bourguignon09}
from the original formulation by Ma, Trudinger, and Wang \cite{MaTrudingerWang05}
in coordinates on the bounded sets $M^\pm \subset \R^n$, as we now recall.

Use subscripts such as $i$ and $j$ to denote derivatives with respect
to $x^i$ and $y^j$,  and commas to separate derivatives
in $\X$ from those 
in $M^-$,  so
that $c_{i,j} = \partial^2 c/\partial x^i \partial y^j$ and
$c_{ij.kl}= \partial^4 c/\partial x^i \partial x^j\partial y^k y^l$, etc.
Also let $c^{k,l}$ denote the matrix inverse of $c_{i,j}$,
and let $D_x c(x,y) = (c_1,c_2,\ldots,c_n)(x,y)$.
Then the original conditions of Ma, Trudinger and Wang were formulated
as the existence of a constant $C_0>0$ such that: \smallskip

\Azero$'$\  $c \in C^4(\bar N)$, and for each
$(x_0,y_0) \in \bar N = \bar M^+ \times \bar M^- \subset \R^{n} \times \R^n$:
\smallskip

\Aone$'_+$ the map $y\in \bar M^- \longmapsto D_x c(x_0,y) \in T_{x_0}^*\X$ is injective;
\smallskip

\Aone$'$  both $c(x,y)$ and $c^*(y,x):= c(x,y)$ satisfy \Azero$'$ and \Aone$'_+$;
\smallskip

\Atwo$'$ $\det c_{i,j}(x_0,y_0) \ne 0$;
\smallskip

\Athree$'_{\mathbf s}$ $(-c_{ij,kl} + c_{ij,m} c^{m,n} c_{kl,n})p^i q^j p^k q^l \ge C_0|p|^2|q|^2$
whenever $p^i c_{i,j} q^j =0$;
\smallskip


\Afour$'$ the sets $D_x c (x_0,M^-) \subset \R^n$ and $D_y c(M^+,y_0) \subset \R^n$ are convex.
\smallskip \\
Here the Einstein summation convention is in effect,  and $|p|$ and $|q|$ denote the
Euclidean norm on $p\in T_{x_0}\X$ and $q \in T_{y_0}\Y \subset \R^n$ respectively.

Their method is heavily based on a priori $C^2$ estimates,
which require a maximum principle for the directional second derivatives
$D^2_{pp}u:=u_{ij}p^i p^j$
of the unknown maximizers $u^\pm \in C(M^\pm)$ for the dual problem \eqref{dual}.
A second-order linear elliptic equation
satisfied by $D^2_{pp}u$ is obtained by twice differentiating the prescribed Jacobian
equation for the map $G$, which is a fully nonlinear
Monge-Amp\`ere type equation for the potential $u=u^+$.
Condition \Athrees$'$ ensures the zeroth order term in the elliptic
equation satisfied by $D^2_{pp}u$ has a coefficient with the correct sign
to admit a maximum principle.

The relaxation \Athree$'$ of $C_0>0$ to $C_0=0$ and strengthening
\Afours$'$ which requires all principal curvatures of $D_x c (x_0,M^-)$ and
$D_y c(M^+,y_0)$ to be positive was introduced in the subsequent investigation of
boundary regularity by
Trudinger and Wang~\cite{TrudingerWang09b}.  We leave it as an exercise to the reader
to confirm the equivalence of each primed hypothesis \Azero$'$-\Afour$'$ and their variants
to the corresponding unprimed hypothesis in the text.  The connection of these conditions to
the Riemann curvature tensor
\begin{equation}\label{Riemann tensor}
\sec^{(\bar N,h)}_{(x_0,y_0)} (p \oplus \0) \wedge (\0 \oplus q) =
(-c_{ij,kl} + c_{ij,m} c^{m,n} c_{kl,n})p^i q^j p^k q^l
\end{equation}
and geodesic equations for the pseudo-metric $h=\frac{1}{2}$Hess$_{(x_0,y_0)}\delta^0$
was first discovered in my joint work with Kim \cite{KimMcCann10}.
However, the link to the cross-difference $\delta^0(x,y)$ originates
in the present work.


\begin{thebibliography}{10}

\bibitem{Ahmad04}
N.~Ahmad.
\newblock {\em The geometry of shape recognition via a Monge-Kantorovich
  optimal transport problem {\tt
  http://www.math.toronto.edu/mccann/ahmad.pdf}}.
\newblock PhD thesis, Brown University, 2004.

\bibitem{AhmadKimMcCann11}
N.~{Ahmad, H.K. Kim, and R.J. McCann}.
\newblock Optimal transportation, topology and uniqueness.
\newblock {\em Bull. Math. Sci.}, 1:13--32, 2011.

\bibitem{AlbertiAmbrosio99}
G.~Alberti and L.~Ambrosio.
\newblock A geometrical approach to monotone functions in {$\R^n$}.
\newblock {\em Math. Z.}, 230:259--316, 1999.

\bibitem{Ambrosio03}
L.~Ambrosio.
\newblock Lecture notes on optimal transport problems.
\newblock In {\em Mathematical Aspects of Evolving Interfaces}, volume 1812 of
  {\em Lecture Notes in Mathematics}, pages 1--52. Springer, Berlin, 2003.

\bibitem{AmbrosioGigli11p}
L.A. Ambrosio and N.~Gigli.
\newblock {\em A user's guide to optimal transport. Preprint.}

\bibitem{BenamouBrenier00}
J.-D. Benamou and Y.~Brenier.
\newblock {A computational fluid mechanics solution to the Monge-Kantorovich
  mass transfer problem}.
\newblock {\em Numer. Math.}, 84(3):375--–393, 2000.

\bibitem{Bourguignon09}
J.-P. Bourguignon.
\newblock Ricci curvature and measures.
\newblock {\em Japan. J. Math.}, 4:25--47, 2009.

\bibitem{Brenier87}
Y.~Brenier.
\newblock D\'ecomposition polaire et r\'earrangement monotone des champs de
  vecteurs.
\newblock {\em C.R. Acad. Sci. Paris S\'er. I Math.}, 305:805--808, 1987.

\bibitem{Brenier91}
Y.~Brenier.
\newblock Polar factorization and monotone rearrangement of vector-valued
  functions.
\newblock {\em Comm. Pure Appl. Math.}, 44:375--417, 1991.

\bibitem{Caffarelli91}
L.A. Caffarelli.
\newblock Some regularity properties of solutions of {Monge-Amp\`ere} equation.
\newblock {\em Comm. Pure Appl. Math.}, 64:965--969, 1991.

\bibitem{Caffarelli96b}
L.A. Caffarelli.
\newblock Boundary regularity of maps with convex potentials --- {II}.
\newblock {\em Ann. of Math. (2)}, 144:453--496, 1996.

\bibitem{CaffarelliFeldmanMcCann00}
L.A. {Caffarelli, M. Feldman and R.J. McCann}.
\newblock {Constructing optimal maps for Monge's transport problem as a limit
  of strictly convex costs}.
\newblock {\em J. Amer. Math. Soc.}, 15:1--26, 2002.

\bibitem{ChampionDePascale11}
T.~Champion and L.~De Pascale.
\newblock The {Monge problem in $\R^d$}.
\newblock {\em Duke Math. J.}, 157:551–--572, 2011.

\bibitem{ChiapporiMcCannNesheim10}
P.-A. {Chiappori, R.J. McCann, and L. Nesheim}.
\newblock Hedonic price equilibria, stable matching and optimal transport:
  equivalence, topology and uniqueness.
\newblock {\em Econom. Theory}, 42:317--354, 2010.

\bibitem{CorderoMcCannSchmuckenschlager01}
D.~Cordero-Erausquin, R.J.~McCann and M.~Schmuckenschl{\"a}ger.
\newblock {A Riemannian interpolation inequality \`a la Borell, Brascamp and
  Lieb}.
\newblock {\em Invent. Math.}, 146:219--257, 2001.

\bibitem{CullenPurser84}
M.J.P Cullen and R.J. Purser.
\newblock An extended {Lagrangian} model of semi-geostrophic frontogenesis.
\newblock {\em J. Atmos. Sci.}, 41:1477--1497, 1984.

\bibitem{Delanoe91}
P.~Delano\"e.
\newblock {Classical solvability in dimension two of the second boundary-value
  problem associated with the Monge-Amp\`ere operator}.
\newblock {\em Ann. Inst. H. Poincar\`e Anal. Non Lin\`eaire}, 8:443--457,
  1991.

\bibitem{DelanoeGe10}
P.~Delano\"e and Y.~Ge.
\newblock Regularity of optimal transportation maps on compact, locally nearly
  spherical, manifolds.
\newblock {\em J. Reine Angew. Math.}, 646:65–--115, 2010.

\bibitem{DelanoeRouviere12p}
P.~Delano\"e and F.~Rouvi\`ere.
\newblock {Positively curved Riemannian locally symmetric spaces are positively
  square distance curved}.
\newblock {\em {\rm To appear in} Canad. J. Math.}

\bibitem{Dudley76}
R.M. Dudley.
\newblock {\em Probabilities and metrics - Convergence of laws on metric
  spaces, with a view to statistical testing}.
\newblock Universitet Matematisk Institut, Aarhus, Denmark, 1976.

\bibitem{EvansGangbo99}
L.C. Evans and W.~Gangbo.
\newblock Differential equations methods for the {M}onge-{K}antorovich mass
  transfer problem.
\newblock {\em {Mem. Amer. Math. Soc.}}, 137:1--66, 1999.

\bibitem{FeldmanMcCann02u}
M.~Feldman and R.J. McCann.
\newblock {Uniqueness and transport density in Monge's transportation problem}.
\newblock {\em Calc. Var. Partial Differential Equations}, 15:81--113, 2002.

\bibitem{Figalli10}
A.~Figalli.
\newblock Regularity properties of optimal maps between nonconvex domains in
  the plane.
\newblock {\em Comm. Partial Differential Equations}, 35:465--479, 2010.

\bibitem{FigalliKim10}
A.~Figalli and Y.-H. Kim.
\newblock {Partial regularity of Brenier solutions of the Monge-Amp\`ere
  equation}.
\newblock {\em Discrete Contin. Dyn. Syst.}, 28:559--–565, 2010.

\bibitem{FigalliKimMcCann-spheres}
A.~Figalli, Y.-H. Kim, and R.J. McCann.
\newblock {Regularity of optimal transport maps on multiple products of
  spheres}.
\newblock {\em {\rm To appear in} J. Euro. Math. Soc. (JEMS)}.

\bibitem{FigalliKimMcCann-econ}
A.~Figalli, Y.-H. Kim, and R.J. McCann.
\newblock {When is multidimensional screening a convex program?}
\newblock {\em J. Econom Theory}, 146:454--478, 2011.

\bibitem{FigalliRifford09}
A.~Figalli and L.~Rifford.
\newblock Continuity of optimal transport maps on small deformations of
  $\mathbb{S}^2$.
\newblock {\em Comm. Pure Appl. Math.}, 62:1670--1706, 2009.

\bibitem{FigalliRiffordVillani12}
A.~{Figalli, L. Rifford and C. Villani}.
\newblock Nearly round spheres look convex.
\newblock {\em Amer. J. Math.}, 134:109, 2012.

\bibitem{ForzaniMaldonado04}
L.~Forzani and D.~Maldonado.
\newblock Properties of the solutions to the {Monge-Amp\`ere} equation.
\newblock {\em Nonlinear Anal.}, 57:815--829, 2004.

\bibitem{Gangbo95}
W.~Gangbo.
\newblock {\em Habilitation thesis}.
\newblock Universit\'e de Metz, 1995.

\bibitem{GangboMcCann96}
W.~Gangbo and R.J. McCann.
\newblock The geometry of optimal transportation.
\newblock {\em Acta Math.}, 177:113--161, 1996.

\bibitem{GangboMcCann00}
W.~Gangbo and R.J. McCann.
\newblock {Shape recognition via Wasserstein distance}.
\newblock {\em Quart. Appl. Math.}, 58:705--737, 2000.

\bibitem{Gigli11}
N.~Gigli.
\newblock {On the inverse implication of Brenier-McCann theorems and the
  structure of $(P_2(M),W_2)$}.
\newblock {\em Methods Appl. Anal.}, 18:127, 2011.

\bibitem{HarveyLawson10p}
F.R. {Harvey and H.B.~Lawson,~Jr.}
\newblock {Split special Lagrangian geometry}.
\newblock {\em Preprint}.

\bibitem{HestirWilliams95}
K.~Hestir and S.C. Williams.
\newblock Supports of doubly stochastic measures.
\newblock {\em Bernoulli}, 1:217--243, 1995.

\bibitem{Kantorovich42}
L.~Kantorovich.
\newblock On the translocation of masses.
\newblock {\em C.R. (Doklady) Acad. Sci. URSS (N.S.)}, 37:199--201, 1942.

\bibitem{KimMcCann08p}
Y.-H. Kim and R.J. McCann.
\newblock {Towards the smoothness of optimal maps on Riemannian submersions and
  Riemannian products (of round spheres in particular)}.
\newblock {\em {\rm Preprint at {\tt arXiv:math/0806.0351v1} To appear in} J.
  Reine Angew. Math.}

\bibitem{KimMcCann10}
Y.-H. Kim and R.J. McCann.
\newblock Continuity, curvature, and the general covariance of optimal
  transportation.
\newblock {\em J. Eur. Math. Soc. (JEMS)}, 12:1009--1040, 2010.

\bibitem{KimMcCannWarren10}
Y.-H. {Kim, R.J. McCann and M. Warren}.
\newblock {Pseudo-Riemannian geometry calibrates optimal transportation}.
\newblock {\em Math. Res. Lett.}, 17:1183--1197, 2010.

\bibitem{KoopmansBeckmann57}
T.C. Koopmans and M.~Beckmann.
\newblock Assignment problems and the location of economic activities.
\newblock {\em Econometrica}, 25:53--–76, 1957.

\bibitem{Lee10}
P.W.Y. Lee.
\newblock {New computable necessary conditions for the regularity theory of
  optimal transportation}.
\newblock {\em SIAM J. Math. Anal.}, 42:3054--3075, 2010.

\bibitem{LeeLi09p}
P.W.Y. Lee and J.~Li.
\newblock {New examples on spaces of negative sectional curvature satisfying
  Ma-Trudinger-Wang conditions}.
\newblock {\em {\rm Preprint at} {\tt arXiv:0911.3978}}.

\bibitem{LeeMcCann11}
P.W.Y. Lee and R.J. McCann.
\newblock {The Ma-Trudinger-Wang curvature for natural mechanical actions}.
\newblock {\em Calc. Var. Partial Differential Equations}, 41:285--299, 2011.

\bibitem{Levin99}
V.L. Levin.
\newblock Abstract cyclical monotonicity and {Monge} solutions for the general
  {Monge-Kantorovich} problem.
\newblock {\em Set-valued Anal.}, 7:7--32, 1999.

\bibitem{Li09}
J.~Li.
\newblock {Smooth optimal transportation on hyperbolic space} {\tt
  www.math.toronto.edu/mccann/papers/li.pdf}.
\newblock Master's thesis, University of Toronto, 2009.

\bibitem{Liu09}
J.~Liu.
\newblock H{\"o}lder regularity of optimal mappings in optimal transportation.
\newblock {\em Calc Var. Partial Differential Equations}, 34:435--451, 2009.

\bibitem{LiuTrudingerWang10}
J.~{Liu, N.S.~Trudinger, X.-J.~Wang}.
\newblock Interior {$C^{2,\alpha}$} regularity for potential functions in
  optimal transportation.
\newblock {\em Comm. Partial Differential Equations}, 35:165--184, 2010.

\bibitem{Loeper09}
G.~Loeper.
\newblock On the regularity of solutions of optimal transportation problems.
\newblock {\em Acta Math.}, 202:241--283, 2009.

\bibitem{Loeper08p}
G.~Loeper.
\newblock {Regularity of optimal maps on the sphere: the quadratic cost and the
  reflector antenna}.
\newblock {\em Arch. Ration. Mech. Anal.}, 199:269--289, 2011.

\bibitem{LoeperVillani10}
G.~Loeper and C.~Villani.
\newblock Regularity of optimal transport in curved geometry: the non-focal
  case.
\newblock {\em Duke Math. J.}, 151:431--485, 2010.

\bibitem{Lorentz53}
G.G. Lorentz.
\newblock An inequality for rearrangements.
\newblock {\em Amer. Math. Monthly}, 60:176--179, 1953.

\bibitem{LottVillani09}
J.~Lott and C.~Villani.
\newblock Ricci curvature for metric measure spaces via optimal transport.
\newblock {\em Annals Math. (2)}, 169:903--991, 2009.

\bibitem{MaTrudingerWang05}
X.-N. {Ma, N. Trudinger and X.-J. Wang}.
\newblock Regularity of potential functions of the optimal transportation
  problem.
\newblock {\em Arch. Rational Mech. Anal.}, 177:151--183, 2005.

\bibitem{McCann95}
R.J. McCann.
\newblock Existence and uniqueness of monotone measure-preserving maps.
\newblock {\em Duke Math. J.}, 80:309--323, 1995.

\bibitem{McCann97}
R.J. McCann.
\newblock A convexity principle for interacting gases.
\newblock {\em Adv. Math.}, 128:153--179, 1997.

\bibitem{McCann99}
R.J. McCann.
\newblock Exact solutions to the transportation problem on the line.
\newblock {\em R. Soc. Lond. Proc. Ser. A Math. Phys. Eng. Sci.},
  455:1341--1380, 1999.

\bibitem{McCann01}
R.J. McCann.
\newblock Polar factorization of maps on {R}iemannian manifolds.
\newblock {\em Geom. Funct. Anal.}, 11:589--608, 2001.

\bibitem{McCannGuillen10p}
R.J. McCann and N.~Guillen.
\newblock Five lectures on optimal transportation: geometry, regularity, and
  applications.
\newblock {\em {\tt http://arxiv.org/abs/1011.2911} To appear in Proceedings of
  the Séminaire de Mathématiques Supérieure (SMS) held in Montréal, QC, June
  27-July 8, 2011}.

\bibitem{McCannSosio11}
R.J. McCann and M.~Sosio.
\newblock {H\"older continuity of optimal multivalued mappings}.
\newblock {\em SIAM J. Math. Anal.}, 43:1855--1871, 2011.

\bibitem{McCannPassWarren12}
R.J. {McCann, B.~Pass and M.~Warren}.
\newblock Rectifiability of optimal transportation plans.
\newblock {\em Canad. J. Math}, 64:924--934, 2012.

\bibitem{Minty62}
G.J. Minty.
\newblock {Monotone (nonlinear) operators in Hilbert space}.
\newblock {\em Duke Math. J.}, 29:341--346, 1962.

\bibitem{Mirrlees71}
J.A. Mirrlees.
\newblock An exploration in the theory of optimum income taxation.
\newblock {\em Rev. Econom. Stud.}, 38:175--208, 1971.

\bibitem{Monge81}
G.~Monge.
\newblock M\'emoire sur la th\'eorie des d\'eblais et de remblais.
\newblock {\em Histoire de l'{A}cad\'emie Royale des Sciences de {P}aris, avec
  les {M}\'emoires de Math\'ematique et de Physique pour la m\^eme ann\'ee},
  pages 666--704, 1781.

\bibitem{Otto01}
F.~Otto.
\newblock The geometry of dissipative evolution equations: The porous medium
  equation.
\newblock {\em Comm. Partial Differential Equations}, 26:101--174, 2001.

\bibitem{OttoVillani00}
F.~Otto and C.~Villani.
\newblock {Generalization of an inequality by Talagrand and links with the
  logarithmic Sobolev inequality}.
\newblock {\em J. Funct. Anal.}, 173:361--400, 2000.

\bibitem{Pass12}
B.~Pass.
\newblock On the local structure of optimal measures in the multi-marginal
  optimal transportation problem.
\newblock {\em Calc. Var. Partial Differential Equations}, 43:529–--536, 2012.

\bibitem{Pratelli07}
A.~Pratelli.
\newblock {On the equality between Monge's infimum and Kantorovich's minimum in
  optimal mass transportation}.
\newblock {\em Ann. Inst. H. Poincar\'e Probab. Statist.}, 43:1–13, 2007.

\bibitem{Pratelli08}
A.~Pratelli.
\newblock On the sufficiency of c-cyclical monotonicity for optimality of
  transport plans.
\newblock {\em Math. Z.}, 258:677–--690, 2008.

\bibitem{RachevRuschendorf98}
S.T. Rachev and L.~R{\"u}schendorf.
\newblock {\em Mass Transportation Problems}.
\newblock Probab. Appl. Springer-Verlag, New York, 1998.

\bibitem{vonRenesseSturm05}
M.-K.~von {Renesse and K.-T. Sturm}.
\newblock {Transport inequalities, gradient estimates, entropy and Ricci
  curvature}.
\newblock {\em Comm. Pure Appl. Math.}, 58:923--940, 2005.

\bibitem{Rochet87}
J.-C. Rochet.
\newblock A necessary and sufficient condition for rationalizability in a
  quasi-linear context.
\newblock {\em J. Math. Econom.}, 16:191--200, 1987.

\bibitem{Rockafellar66}
R.T. Rockafellar.
\newblock Characterization of the subdifferentials of convex functions.
\newblock {\em Pacific J. Math.}, 17:497--510, 1966.

\bibitem{RuschendorfRachev90}
L.~R{\"u}schendorf and S.T. Rachev.
\newblock A characterization of random variables with minimum {$L^2$}-distance.
\newblock {\em J. Multivariate Anal.}, 32:48--54, 1990.

\bibitem{SchachermayerTeichmann09}
W.~Schachermayer and J.~Teichmann.
\newblock {Characterization of optimal transport plans for the
  Monge-Kantorovich problem}.
\newblock {\em Proc. Amer. Math. Soc.}, 137:519–--529, 2009.

\bibitem{SmithKnott87}
C.~Smith and M.~Knott.
\newblock On the optimal transportation of distributions.
\newblock {\em J. Optim. Theory Appl.}, 52:323--329, 1987.

\bibitem{Spence73}
M.~Spence.
\newblock Job market signaling.
\newblock {\em Quarterly J. Econom.}, 87:355--374, 1973.

\bibitem{Sturm06ab}
K.-T. Sturm.
\newblock {On the geometry of metric measure spaces, I and II}.
\newblock {\em Acta Math.}, 196:65--177, 2006.

\bibitem{TrudingerWang01}
N.S. Trudinger and X.-J. Wang.
\newblock {On the Monge mass transfer problem}.
\newblock {\em Calc. Var. Paritial Differential Equations}, 13:19--31, 2001.

\bibitem{TrudingerWang09b}
N.S. Trudinger and X.-J. Wang.
\newblock {On the second boundary value problem for Monge-Amp\`ere type
  equations and optimal transportation}.
\newblock {\em Ann. Sc. Norm. Super. Pisa Cl. Sci. (5)}, 8:1--32, 2009.

\bibitem{Urbas97}
J.~Urbas.
\newblock {On the second boundary value problem for equations of Monge-Amp\`ere
  type}.
\newblock {\em J. Reine Angew. Math.}, 487:115--124, 1997.

\bibitem{Villani03}
C.~Villani.
\newblock {\em Topics in Optimal Transportation}, volume~58 of {\em Graduate
  Studies in Mathematics}.
\newblock American Mathematical Society, Providence, 2003.

\bibitem{Villani09}
C.~Villani.
\newblock {\em Optimal Transport. Old and New}, volume 338 of {\em Grundlehren
  der Mathematischen Wissenschaften [Fundamental Principles of Mathematical
  Sciences]}.
\newblock Springer, New York, 2009.

\bibitem{Wang96}
X.-J. Wang.
\newblock {On the design of a reflector antenna}.
\newblock {\em Inverse Problems}, 12:351--375, 1996.

\bibitem{Wang04}
X.-J. Wang.
\newblock {On the design of a reflector antenna II}.
\newblock {\em Calc. Var. Partial Differential Equations}, 20:329--341, 2004.

\bibitem{Yu07}
Y.~Yu.
\newblock Singular set of a convex potential in two dimensions.
\newblock {\em Comm. Partial Differential Equations}, 32:1883–--1894, 2007.

\end{thebibliography}


\end{document}